\documentclass[12pt, a4paper]{article}
\usepackage{latexsym,amssymb,amsfonts,amsmath,epsfig,tabularx}

\usepackage{hyperref}

\usepackage{graphics}

\usepackage{verbatim}

\usepackage{enumitem} 

\usepackage[usenames]{xcolor}
\newif\ifmarkauthors
\markauthorstrue
\ifmarkauthors
  \definecolor{darkred}{RGB}{139,0,0}
  \definecolor{darkgreen}{RGB}{0,100,0}
  \definecolor{darkmagenta}{RGB}{139,0,139}
  \definecolor{darkorange}{RGB}{190,70,20}
  \newcommand{\gw}[1]{{\color{red}{#1}}}

  \def\cbdelete[#1]{}
\else
  \newcommand{\gw}[1]{{#1}}

  \def\cbdelete[#1]{}
\fi

\newtheorem{theorem}{Theorem}
\newtheorem{remark}[theorem]{Remark}

\newtheorem{proposition}[theorem]{Proposition}

\newtheorem{example}[theorem]{Example}

\gw{\newtheorem{test}[theorem]{Numerical Test}}

\def\rho{\varrho}      

\def\bbN{\mathbb{N}}
\def\e{\varepsilon}
\newcommand{\abs}[1]{\left\vert#1\right\vert}
\newcommand{\norm}[1]{\left\Vert#1\right\Vert}
\newcommand{\esup}{\operatornamewithlimits{ess\,sup}}

\newcommand{\setu}{{\mathfrak{u}}}

\newenvironment{proof}{\begin{trivlist}\item[\hskip\labelsep{\bf Proof.}]}{$\hfill\Box$\end{trivlist}}

\newcommand{\rd}{{\mathrm{\,d}}}

\newcommand{\calF}{{\mathcal{F}}}

\newcommand{\bsz}{{\boldsymbol{z}}}

\newcommand{\bsx}{{\boldsymbol{x}}}

\newcommand{\bst}{{\boldsymbol{t}}}

\newcommand{\bszero}{{\boldsymbol{0}}}

\newcommand{\bbR}{{\mathbb{R}}}

\newcommand{\mask}[1]{}

\title{On efficient weighted integration via a change of variables}

\author{P. Kritzer\thanks{P. Kritzer is supported by the Austrian
Science Fund (FWF):
Project F5506-N26, which is a part of the Special Research Program
"Quasi-Monte Carlo Methods:
Theory and Applications".}, F. Pillichshammer\thanks{F. Pillichshammer is
 supported by the Austrian Science Fund (FWF): Project F5509-N26,
which is a part of the Special Research Program "Quasi-Monte Carlo Methods:
Theory and Applications".}, L. Plaskota\thanks{L. Plaskota is supported by 
the National Science Centre, Poland, under grant 2017/25/B/ST1/00945.},
and G. W. Wasilkowski}

\date{}

\begin{document}
\maketitle

\begin{abstract}
\noindent In this paper, we study the approximation of $d$-dimensional $\rho$-weighted
integrals over unbounded domains $\bbR_+^d$ or $\bbR^d$ 
using a special change of variables, so that quasi-Monte Carlo (QMC) or sparse grid rules can be applied 
to the transformed integrands over the unit cube. We consider a class of integrands with
bounded $L_p$ norm of mixed partial derivatives of first order, where $p\in[1,+\infty].$

The main results give sufficient conditions on the change of
variables $\nu$ which guarantee that the transformed integrand belongs to the standard Sobolev space 
of functions over the unit cube  with mixed smoothness of order one. These conditions depend on $\rho$ and $p$.

The proposed change of variables is in general different than the standard change based on the inverse of the cumulative
distribution function. We stress that the standard change of variables
leads to integrands over a cube; however, those integrands have 
singularities which make the application of QMC and sparse grids ineffective.
Our conclusions are supported 
by numerical experiments.
\end{abstract}

\section{Introduction}
We consider in this paper the approximation of $d$-variate
$\rho_d$-weighted integrals of the form 
\begin{equation}\label{MultiInt}
I_{d,\rho}(f)\,=\,\int_{D^d}f(\bsx)\,\rho_d(\bsx)\rd\bsx,\quad
\mbox{where}\quad \rho_d(\bsx)\,=\,\prod_{j=1}^d\rho(x_j),
\end{equation}
over an unbounded domain $D^d$, where
\[D=\bbR_+\quad\mbox{or}\quad D=\bbR
\]
and for a given probability density function
\[
\rho:D\to\bbR_+.
\]
Such integrals in the univariate case are often approximated by
Gaussian quadratures that enjoy exponential rate of convergence,
see, e.g., \cite{DavRab1984}. There are also generalized Gaussian rules,
see, e.g., \cite{GrieOett14} and the papers cited there, that achieve
exponential rate for integrands with singularities
at infinity. We stress that those results are about the asymptotic behavior of the  
integration error and require analytic integrands. In the current  
paper, we consider integrands of regularity one only, and we analyze the  
worst case error with respect to a class of integrands.

Indeed, we follow the
{\em Information-Based Complexity} approach (see, e.g., \cite{TWW88})
providing worst case results for all integrands $f$ from the Sobolev space
$F_{d,p}(D^d)$ of functions defined on $D^d$ with
mixed partial derivatives of first order bounded in the
standard $L_{p}(D^d)$ space, where $p\in[1,+\infty].$
As will be clear later, our results may also be applied to 
$\gamma$-weighted spaces of $\infty$-variate functions (i.e., $d=+\infty$)
as well as
for $\rho_d$ being a product of different probability densities
(i.e., $\rho(\bsx)=\prod_{j=1}^d\rho_j(x_j)$).

The spaces $F_{d,p}(D^d)$ have been considered in a number of papers, and 
they are naturally related to the Sobolev spaces $W_{d,p}(B^d)$
of functions on a unit
cube $B^d$ with mixed partial first order derivatives in
$L_{p}(B^d)$. Indeed, a quite common 
approach is to use the change of variables
$x:=\Phi_\rho^{-1}(t)$, where $\Phi_\rho$ is the
{\em cumulative distribution function (CDF)} for the probability density
$\rho$, to reduce the $\rho$-weighted integrals over
$D^d$ to the standard Lebesgue integration over a unit cube $B^d$,
and next apply algorithms that are efficient for 
spaces $W_{d,p}(B^d)$. However, as we will show, such an
approach is well founded for only $p=1,$ since for
$p>1$ the change of variables based on 
the inverse of the CDF produces integrands with boundary singularities, and so they {\bf do not} belong to the spaces $W_{d,p}(B^d)$.

To simplify the notation, we will denote $W_{d,p}(B^d)$ and
$F_{d,p}(D^d)$, respectively, by
\[
  W_{d,p}\quad\mbox{and}\quad F_{d,p}.
\]
We investigate changes of variables that transform
the weighted integrals $I_{d,\rho}(f)$ over unbounded $D^d$ into
standard Lebesgue integrals over a bounded cube $B^d$:
\[
I_d(g_{f,\nu})\,=\,\int_{B^d} g_{f,\nu}(\bst)\rd\bst\quad\mbox{with}
\quad g_{f,\nu}(\bst)\,=\,f(\nu(t_1),\dots,\nu(t_d))\,
\prod_{j=1}^d\left(\rho(\nu(t_j))\,\nu'(t_j)\right).
\]
We are searching for {\em change of variables functions}
$\nu:B \rightarrow D$ such that the obtained integrands satisfy
\begin{equation}\label{dupa-1}
  g_{f,\nu}\,\in\,W_{d,p}\quad\mbox{for all\ }f\,\in\,F_{d,p}.
\end{equation}

We now explain the significance of the requirement \eqref{dupa-1}.
If the integrands $g_{f,\nu}$ satisfy \eqref{dupa-1}, then their
integrals can be approximated by cubatures that are known to
have small worst case errors with respect to the spaces $W_{d,p}$,
including {\em quasi-Monte Carlo (QMC)} and {\em sparse grid} methods.
This in turn provides efficient cubatures for the original
$\rho_d$-weighted integration over unbounded~$D^d$. 
More precisely, let
\[ {\rm error}(Q_{d,n};W_{d,p})\,=\,\sup_{g\in W_{d,p}}
\frac{\left|\int_{B^d}g(\bst)\rd\bst-Q_{d,n}(g)\right|}
{\|g\|_{W_{d,p}}}
\]
be the worst case error of a cubature $Q_{d,n}$.
Then the cubature $Q_{d,\rho,n}$ for the weighted integrals $I_{d,\rho}$
over unbounded domains $D^d$, given by
\begin{equation}\label{dupa-2}
  Q_{d,\rho,n}(f)\,:=\,Q_{d,n}(g_{f,\nu}),
\end{equation}
has the worst case error
\[{\rm error}(Q_{d,\rho,n};F_{d,p})\,=\,
\sup_{f\in F_{d,p}}\frac{\left|I_{d,\rho}(f)-Q_{d,\rho,n}(f)\right|}
{\|f\|_{F_{d,p}}},
\]
bounded by
\begin{equation}\label{dupa-3}
{\rm error}(Q_{d,\rho,n};F_{d,p})\,\le\,
{\rm error}(Q_{d,n};W_{d,p})\,\cdot\,
C_{d,p}(\nu),
\end{equation}
where
\[C_{d,p}(\nu)\,:=\,\sup_{f\in F_{d,p}}
\frac{\|g_{f,\nu}\|_{W_{d,p}}}{\|f\|_{F_{d,p}}}.
\]
This is why, in our search, we are looking for functions $\nu$ with finite and possibly small $C_{d,p}(\nu)$.

As already mentioned, the most common
change of variables that uses the inverse of the CDF, 
\begin{equation}\label{dupa-4}
  \nu(t)\,:=\,\Phi^{-1}_\rho(t)\quad\mbox{for}\quad
  \Phi_\rho(x)\,=\,\int_{y\le x}\rho(y)\rd y,
\end{equation}
gives $C_{d,p}(\nu)=+\infty,$ for all $p>1$ and 
{\em non-trivial} densities $\rho$, see Proposition \ref{prop:std}.
Moreover, it produces functions
$g_{f,\nu}$ that have singularities (poles) on the boundary of the cube
$B^d,$ even for exceptionally smooth $f$. We illustrate this by the following
example. Let $D=\bbR_+$, $\rho(x)=\exp(-x)$, $p=+\infty$, and $f(x)=x$.
Then $\|f\|_{F_{1,\infty}}=1,$ yet for $\nu$ as in \eqref{dupa-4}
the change of variables produces 
\[
g_{f,\nu}(t)\,=\,-\,\ln(1-t),
\]
which has a singularity at $t=1.$ Using instead
\[
\nu_{a}(x)\,=\,a\,\Phi^{-1}_\rho(x)\quad\mbox{with}\quad a>1
\]
results in
\[g_{f,\nu_{a}}(t)\,=\,-a^2\,(\ln(1-t))\,(1-t)^{a-1},
\]
which even belongs to $W_{1,\infty}$ for $a>2$ (compare with Fig.\ref{fexpo}). 
\begin{figure}
\begin{center}
\includegraphics[width=12cm]{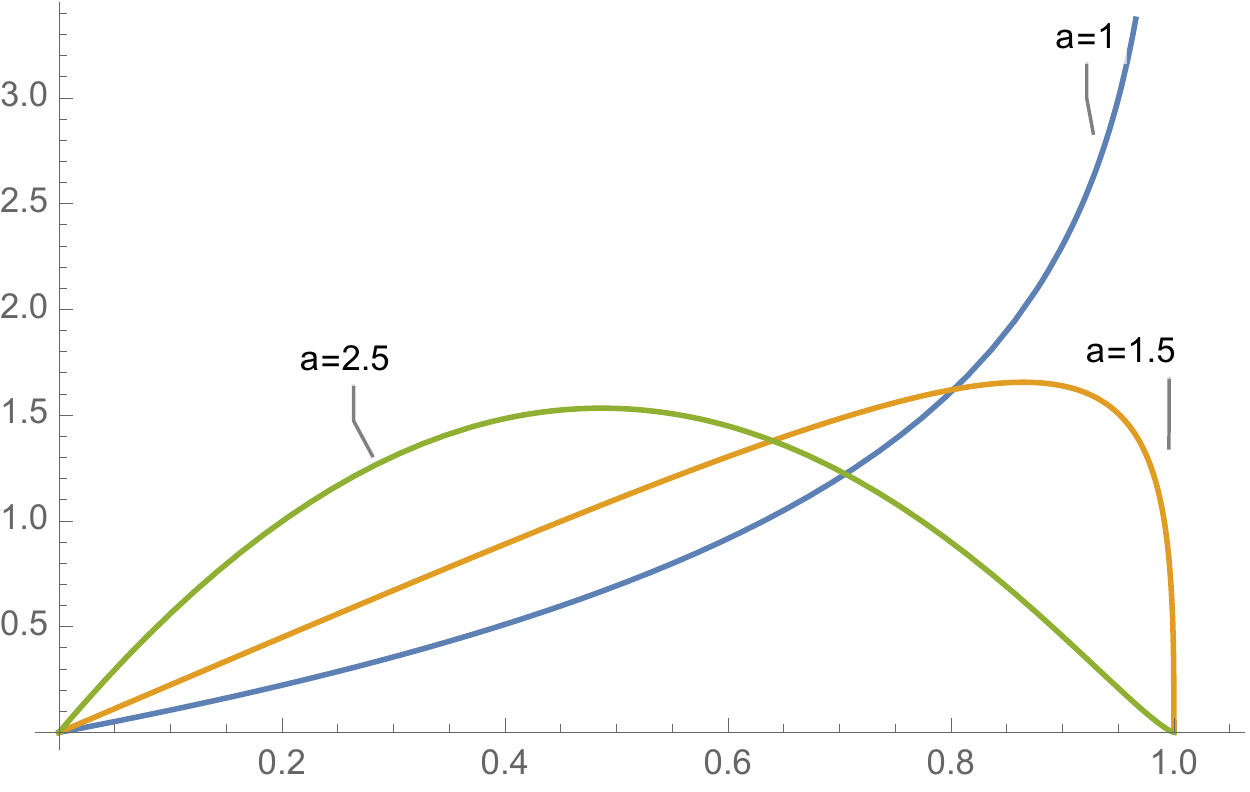}
\caption{Plot of $g_{f,\nu_a}(t)=-a^2(\ln(1-t))(1-t)^{a-1}$ for 
$a\in \{1, 1.5, 2.5\}$.}
\label{fexpo}
\end{center}
\end{figure}

Moreover, the midpoint rule with $10^5$ samples applied to $g_{f,\nu_a}$ with
$a\in\{1, 1.5, 2.5\}$ produces, correspondingly, approximations with absolute errors 
$$ 3.47\times10^{-6},\qquad 5.02\times10^{-8},\qquad 2.65\times10^{-11}, $$
see Numerical Test \ref{Test1} for more details.

In Theorems \ref{thm:main} and \ref{prop:main}, we provide
conditions on the change of variables function $\nu$, that
guarantee finite $C_{d,p}(\nu)$. It turns out that for a number
of specific probability density functions, including
$\rho(x)=\lambda^{-1}\,\exp(-x/\lambda)$ with $D=\bbR_+$
and 
$\rho(x)=(2\pi\sigma^2)^{-1/2}\,\exp(-x^2/(2\,\sigma^2))$
with $D=\bbR$, these conditions are satisfied by
$\nu_a=a\,\Phi_\rho^{-1}$ for a specially chosen $a\ge 1$. However,
with the exception of $p=1$, we still have $C_{1,p}(\nu_1)=+\infty$
for $a=1$. Moreover, as shown in Example~\ref{ex:6},
$\nu_a=a\,\Phi^{-1}_\rho$ need not yield finite $C_{d,p}(\nu_a)$
for any $a$ (unless $p=1$) for a special probability density
function $\rho$. In this case, we provide another change of variable
function $\nu$ with finite $C_{d,p}(\nu)$. 

Finally, we add that the results of this paper can be used to produce
an efficient implementation of the {\em Multivariate Decomposition Method (MDM)}
for approximation of weighted integrals of functions that belong to
$\gamma$-weighted spaces $\calF_{d,p,\gamma}$ for any $d$
including $d=+\infty$. Due to a number of specific details, we
delay the discussion of MDM's to Section \ref{Sect:MDM}.

\section{Multivariate Integration}
In the multivariate integration problem \eqref{MultiInt},
we assume that the integrands $f$ belong to the Sobolev space
$F_{d,p}$ of 
functions anchored at zero and whose mixed first order partial 
derivatives (in the weak sense) are in the $L_p$ space.
By anchored we mean that
\[f(\bsx)\,=\,0\quad\mbox{if}\quad x_j\,=\,0 \quad\mbox{for some\ }
j\in\{1,\dots,d\}.
\]
This assumption is not restrictive, as will be explained in
Section \ref{Sect:MDM}.

Such spaces were introduced for $d=1$ in \cite{WaWo2000} and have been 
later considered in a number of papers.
Specifically, see, e.g. \cite{GnHeHiRiWa16}, a function $f$ is in
$F_{d,p}$ iff it is of the form
\begin{equation}\label{prod-rep}
f(\bsx)\,=\,\int_{D^{d}}h_f(\bsz)\,\prod_{j=1}^d \kappa(x_j,z_j)
\rd\bsz,
\end{equation}
where 
\[\kappa(x,z)\,=\,\left\{\begin{array}{rl} 1 &\mbox{if\ }
x>z\ge 0,\\ 
-1 &\mbox{if\ }x<z <0,\\
0  &\mbox{otherwise}.\end{array}\right.
\]
The mixed first order partial derivative 
\[
\prod_{j=1}^d\frac{\partial}{\partial x_j}f\,
:=\frac{\partial^d}{\partial x_1 \partial x_2 \cdots \partial x_d}f
\]
is then equal to $h_f$ and is assumed to be in the $L_p=L_p(D^d)$ space,
$p\in [1,+\infty]$, which is endowed with the norm
\[\|h_f\|_{L_p(D^d)}\,:=\, 
  \left(\int_{D^d}|h_f(\bsx)|^p\rd\bsx\right)^{1/p}
    \quad\mbox{for}\;\;1\le p<+\infty,
\]
and 
\[\|h_f\|_{L_\infty(D^d)}\,:=\,
\esup_{\bsx\in D^d}|h_f(\bsx)|\quad \mbox{for\ }p=+\infty. 
\]
Then the norm in $F_{d,p}$ is given by 
\[
     \|f\|_{F_{d,p}}\,=\,\|h_{f}\|_{L_p(D^d)}\,=\,\left\|
\prod_{j=1}^d\frac{\partial}{\partial x_j}\,f\right\|_{L_p(D^d)},
\]
which is well defined since the functions $f\in F_{d,p}$
are anchored at zero.

Throughout the paper, $p^*$ denotes the conjugate of $p$, i.e.,
\[\frac1p+\frac1{p^*}\,=\,1.\]
Due to \eqref{prod-rep} and H\"older's inequality, for
$1<p\le+\infty$ we have
\begin{eqnarray*}
\left|\int_{D^d}f(\bsx)\,\rho_d(\bsx)\rd\bsx\right|&=&
\left|\int_{D^d}h_f(\bsz)\int_{D^d}\prod_{j=1}^d
\left(\kappa(x_j,z_j)\,\rho(x_j)\right)\rd\bsx\rd\bsz\right|\\
&\le&\|f\|_{F_{d,p}}\,  \left(\int_{D^d}\,
\left|\int_{D^d}\prod_{j=1}^d(\kappa(x_j,z_j)\,\rho(x_j))
\rd\bsx\right|^{p^*}\rd\bsz\right)^{1/p^*}.
\end{eqnarray*}
Since H\"older's inequality is sharp, we conclude that if the
integration functional $I_{d,\rho}$ is bounded, then its operator
norm is given by
\[
\|I_{d,\rho}\| \,=\,  \left(\int_{D^d}
\left|\int_{D^d}\prod_{j=1}^d(\kappa(x_j,z_j)\,\rho(x_j))
\rd\bsx\right|^{p^*}\rd\bsz\right)^{1/p^*}\,=\,\|I_{1,\rho}\|^{d}
\]
with
\[\|I_{1,\rho}\|\,=\,\left(\int_D\left| 
\int_D \kappa(x,z)\,\rho(x)\rd x\right|^{p^*}\rd z\right)^{1/p^*}.
\]
Similarly, for $p=1$ the equality $\|I_{d,\rho}\|=\|I_{1,\rho}\|^d$
holds with
\[\|I_{1,\rho}\|\,=\,\esup_{z\in D}\left| 
\int_D \kappa(x,z)\,\rho(x)\rd x\right|.
\] 

We propose using a change of variables that transforms the original weighted integral 
over the unbounded domain $D^d$ to the standard Lebesgue integral
over the unit cube $B^d$  with the new integrand belonging to
the standard Sobolev space $W_{d,p},$ as required in \eqref{dupa-1},
and then use,  e.g., QMC or
sparse grid methods, which are well suited and quite
efficient for such kinds of integrands.
Also here we assume, without any loss of generality,
that functions $g\in W_{d,p}$ are anchored at zero.

Specifically, we apply (componentwise) a change of variables
\[
x_j\,=\,\nu(t_j) \quad\mbox{for}\quad1\le j\le d,
\]
where the function $\nu:B\to D$ is monotonically increasing and onto $D,$ and
\[ B\,=\,\left\{\begin{array}{ll} [0,1) & \mbox{if\ }D=\bbR_+,\\
  (-1/2,1/2) & \mbox{if\ }D=\bbR. 
\end{array}\right.
\]
Then, assuming that the derivative $\nu'$ exists and is measurable, we obtain
\begin{equation}\label{trnsf}
  I_{d,\rho}(f)\,=\,\int_{B^d} g(\bst)\rd\bst
\qquad\mbox{with}\qquad
g(\bst)\,\,:=\,f(\nu(t_1),\dots,\nu(t_d))\,
  \prod_{j=1}^d \rho(\nu(t_j))\,\nu'(t_j). 
\end{equation}

To stress the dependence of $g$ on the functions
$f$ and $\nu$, we will often write
\[\mbox{$g_{f,\nu}$ or $g_f$ instead of $g$.}
\]

We are interested in functions $\nu$ such that the corresponding
integrands $g=g_{f,\nu}$ are in 
$W_{d,p}$ and are
anchored at zero.
That is, 
\[
\|g\|_{W_{d,p}}\,=\,\left[\int_{B^d}\left|\prod_{j=1}^d
\frac{\partial}{\partial x_j} g(\bsx)\right|^p\rd \bsx\right]^{1/p}
\,<\,+\infty\quad\mbox{and $g(\bsx)=0$ if $x_j=0$ for some $j$}
\]
(with the obvious modification for $p=+\infty$).
Moreover, we would like the ratio 
\begin{equation}\label{ratio:C}
  C_{d,p}(\nu)\,:=\,\sup_{f\in F_{d,p}}
  \frac{\|g_f\|_{W_{d,p}}}{\|f\|_{F_{d,p}}}
\end{equation}
to be not only finite, but also small, as explained in the introduction,
see \eqref{dupa-3}.
Due to the tensor product form \eqref{prod-rep}, one can show 
(in a similar way to the derivation of the norm of $I_{d,\rho}$) that the
following holds.

\begin{proposition}\label{prop:univar}
The ratio $C_{d,p}(\nu)$ is fully determined by the univariate case, i.e., 
\[
C_{d,p}(\nu)\,=\,\left(C_{1,p}(\nu)\right)^d,
\qquad\mbox{where}\qquad
C_{1,p}(\nu)\,=\,\sup_{f\in F_{1,p}}
\frac{\|g_f\|_{W_{1,p}}}{\|f\|_{F_{1,p}}}.
\]
\end{proposition}

This is why, in the analysis of the factor $C_{d,p}(\nu),$
we can restrict our considerations to the univariate case.

\begin{remark}\rm
Denote by $J_d$ the standard Lebesgue integral over the unit
cube $B^d$, i.e.,
\[J_d(g)\,=\,\int_{B^d} g(\bst)\rd \bst.
\]
Then its operator norm equals $\|J_d\|=\|J_1\|^d,$ where
\[
\|J_1\| = \left\{ \begin{array}{ll} 
                    (1+p^*)^{-1/p^*} &\quad \mbox{if}\ 1<p\le+\infty,\\
                       1 &\quad\mbox{if}\ p=1, \end{array}\right.
\] 
for $B=[0,1)$, and is twice smaller for $B=(-1/2,1/2)$. Furthermore, 
$$ |I_{d,\rho}(f)|=\bigg|\int_{B^d}g_f(\bst)\rd\bst\bigg|\le\|J_d\|\|g_f\|_{W_{d,p}}
     \le \|J_d\|C_{d,p}(\nu)\|f\|_{F_{d,p}}=\big(\|J_1\|C_{1,p}(\nu)\big)^d\|f\|_{F_{d,p}}. $$
Hence, if  $C_{1,p}(\nu)<+\infty$ for some $\nu$,
then the operator $I_{d,\rho}$  
is bounded in $F_{d,\rho},$ and its norm satisfies
$$  \|I_{d,\rho}\|\,\le\,\Big(\|J_1\|\, C_{1,p}(\nu)\Big)^d.$$
\end{remark}

\begin{remark}\rm
When $\rho_d$ is a product of different weights depending on $j$, then  
also $\nu$ being a product of different $\nu_j$ should be considered.
That is, if $\rho_d(\bsx)\,=\,\prod_{j=1}^d\rho_{1,j}(x_j)$, 
then one could consider $\nu(\bst)\,=\,\prod_{j=1}^d\nu_{j}(t_j)$. 
Then $C_{d,p}(\nu)\,=\,\prod_{j=1}^dC_{1,p}(\nu_{j})$.
\end{remark}

Before we propose our ways of changing variables, 
let us see what happens when we apply the standard change  
that has been commonly used in practice.

\paragraph{Standard Change of Variables.}
Let $\Phi_\rho :\bbR_+ \rightarrow [0,1)$ or 
$\Phi_\rho :\bbR \rightarrow (0,1)$ 
be the CDF corresponding to the probability density $\rho,$ defined by 
\begin{equation}\label{eq:cdf}
  \Phi_\rho(x)\,=\,\int_D\,1_{\{x\ge t\}}\,\rho(t) \rd t,
\end{equation}
where $1_A$ denotes the indicator function of a given set $A$. Then the standard change of variables uses
$\nu\,=\,\Phi^{-1}_\rho$, i.e.,
\begin{equation}\label{stndrd}
  \int_{D^d}f(\bsx)\,\rho(\bsx)\rd\bsx\,=\,\int_{[0,1]^d}
  g_{f,\Phi^{-1}_\rho}(\bst)\rd\bst,
\quad\mbox{where}\quad
g_{f,\Phi^{-1}_\rho}(\bst)\,=\,f\left(\Phi^{-1}_\rho(t_1),\dots,
\Phi^{-1}_\rho(t_d)\right). 
\end{equation}
We have the following simple yet important proposition.

\begin{proposition}\label{prop:std}
For any $f\in F_{d,p}$, let $g_{f,\Phi^{-1}_\rho}$ be the function given by
\eqref{stndrd}.   
\begin{description}
\item{(i)} For $p=1$, we have
\[\|f\|_{F_{d,1}}\,=\,\|g_{f,\Phi^{-1}_\rho}\|_{W_{d,1}}\quad\mbox{for all\ }
f\in F_{d,1}.  
\]  
\item{(ii)} For $p>1$, we have
\begin{equation}\label{eq:std-dupa}
\|g_{f,\Phi^{-1}_\rho}\|_{W_{d,p}}\,=\,
\left[\int_{D^d}\left|\left(\prod_{j=1}^d\frac\partial{\partial x_j} f(\bsx)\right)
\left(\prod_{j=1}^d\frac1{\rho^{1/p^*}(x_j)}\right)\right|^p
\rd\bsx\right]^{1/p}.
\end{equation}
Hence $g_{f,\Phi^{-1}_\rho}\in W_{d,p}$ if and only if the right-hand side
of \eqref{eq:std-dupa} is finite. 
\end{description}
\end{proposition}

\begin{proof}
Due to the tensor product structure of the spaces, it is enough to 
prove the proposition for $d=1$. We clearly have 
\[
g'_{f,\Phi_\rho}(t)\,=\,f'(\Phi^{-1}_\rho(t))\,
\frac{\rd}{\rd t}\Phi^{-1}_\rho(t)
\,=\,f'(\Phi^{-1}_\rho(t))\,\frac1{\rho(\Phi^{-1}_\rho(t))}.
\]
Hence for $1\le p<+\infty$ we have
\begin{eqnarray*}
 \|g_{f,\Phi^{-1}_\rho}\|_{W_{1,p}}^{p}&=&
  \int_0^1|g'_{f,\Phi^{-1}_\rho}(t)|^p\rd t \, = \,
\int_0^1 \left|f'(\Phi^{-1}_\rho(t))\,
\frac1{\rho(\Phi^{-1}_\rho(t))}\right|^p\rd t \\
&=&\int_D\left|f'(x)\,\frac1{\rho(x)}\right|^p\,\rho(x)\,\rd x
\,=\,\int_D\left|f'(x)\,\frac{1}{\rho^{1/p^*}(x)}\right|^p\rd x, 
\end{eqnarray*}
(with $1/p^*=0$ for $p=1$), 
and for $p=+\infty$ we have
\begin{equation*}
  \|g_{f,\Phi^{-1}_\rho}\|_{W_{1,\infty}}=
  \esup_{t\in [0,1]}|g_{f,\Phi^{-1}_\rho}'(t)|
=\esup_{t\in [0,1]}\left|\frac{f'(\Phi_\rho^{-1}(t))}{\rho(\Phi_\rho^{-1}(t))}
\right| =\esup_{x\in D}\left|\frac{f'(x)}{\rho(x)}\right|,
\end{equation*}
as claimed. 
\end{proof}

Proposition \ref{prop:std} states that the change of variables
\eqref{stndrd} transforms all functions from $F_{d,p}$ to $W_{d,p}$
only for $p=1$. 
For $p>1$, the situation is quite different. Then $1/p^*>0$ and,
therefore, only 
those functions whose mixed first order partial derivatives converge
to zero sufficiently fast have the corresponding functions
$g_{f,\Phi^{-1}_\rho}$ in the space $W_{d,p}$. Indeed, since the weight
$\rho(x)$ has to converge to zero as $|x|\to\infty$, $\rho^{-1/p^*}(x)$
has to converge to infinity, resulting in a very restrictive class of
integrands $f$.
In particular, if $p^*=1$ then we need to have
\[
\esup_{\bsx\in D^d}\left|\frac{\prod_{j=1}^d\frac\partial{\partial x_j}
  f(\bsx)}{\prod_{j=1}^d\rho(x_j)}\right|\,<\,+\infty. 
\]
For instance, for $D=\bbR$ and Gaussian weight
\[
\rho(x)=\frac{1}{\sqrt{2\pi\sigma^2}}\,\exp\left(-\frac{x^2}{2\sigma^2}\right)
\] 
we would need 
\[
 \int_{\bbR^d} \exp\bigg(\frac{1}{2\sigma^2p^*}\sum_{j=1}^dx_j^2\bigg)
  \bigg|\prod_{j=1}^d\frac{\partial}{\partial x_j}f(\bsx)\bigg|^p\rd\bsx
   \,<\,+\infty \qquad\mbox{for\ }p\,<\,+\infty,
\]
and 
\[
\esup_{\bsx\in\bbR^d}\;  
\exp\bigg(\frac{1}{2\sigma^2}\sum_{j=1}^dx_j^2\bigg)
  \bigg|\prod_{j=1}^d\frac{\partial}{\partial x_j}f(\bsx)\bigg|\,<\,+\infty
  \qquad\mbox{for\ }p\,=\,+\infty,
\]
which may hold only if the derivative of $f$ decreases to zero faster than exponentially. This is why in the rest of the paper we concentrate on the case of $p>1$.

\smallskip
We now switch to the univariate case, but will return to the multivariate case 
in Section \ref{Sect:MDM} with some additional comments.

\section{Univariate Functions}
Recall that for the univariate case the space $F_{1,p}$ 
consists of functions anchored at zero and whose weak derivatives 
are in the $L_p(D)$ space. For $f\in F_{1,p}$ we want to know when 
the corresponding functions 
\begin{equation}\label{eq:change}
 g_f(t)\,=\,g_{f,\nu}(t)\,=\,f(\nu(t))\,\rho(\nu(t))\,\nu'(t)
\end{equation}
belong to the standard Sobolev space $W_{1,p}$ of functions
anchored at zero with $g'\in L_{p}(B)$. 

The following functions will play an important role:
\begin{eqnarray}
  h_{0,p^*}(t) &:=& \rho(\nu(t))\,\nu'(t)\,|\nu(t)|^{1/p^*},
  \label{H0} \\
h_{1,p^*}(t) &:=&\rho(\nu(t))\,(\nu'(t))^{1+1/p^*},
\label{H1} \\  
h_{2,p^*}(t) &:=& \left(\rho(\nu(t))\,\nu'(t)\right)'\,    
|\nu(t)|^{1/p^*}. \label{H2}
\end{eqnarray}
For this to make sense we obviously assume that the corresponding 
functions $\rho$ and $\nu$ are sufficiently regular. 

\subsection{Case of $D=\bbR_+$}
Let $\nu:[0,1)\to\bbR_+$ be an increasing and twice differentiable 
function such that
\[\nu(0)\,=\,0\quad\mbox{and}\quad\lim_{t\to1}\nu(t)\,=\,+\infty. 
\]

\begin{theorem}\label{thm:main}
For every $f\in F_{1,p}$, the corresponding function $g_f$
given by \eqref{eq:change} satisfies:
\begin{description}
\item{(i)} $\|g_f\|_{L_\infty(0,1)}<+\infty$\quad if \quad
$\|h_{0,p^*}\|_{L_\infty(0,1)}< +\infty$,
\item{(ii)} $g_f\in W_{1,p}$ if
$\|h_{1,p^*}\|_{L_{\infty}(0,1)}<+\infty$ and 
$\|h_{2,p^*}\|_{L_p(0,1)}\,<\,+\infty$.
If so, then 
\[
C_{1,p}(\nu)\,\le\,\|h_{1,p^*}\|_{L_\infty(0,1)}+
\|h_{2,p^*}\|_{L_p(0,1)}.
\]
In particular, for $p=+\infty$, 
\[
C_{1,\infty}(\nu)\,\le\,
\left\|\rho(\nu(\cdot))\,(\nu'(\cdot))^2\right\|_{L_\infty(0,1)}+
  \left\|(\rho(\nu(\cdot))\,\nu'(\cdot))'\,\nu(\cdot)
  \right\|_{L_\infty(0,1)}.
\]
\end{description}
\end{theorem}

\begin{proof}
To show {\it (i)}, observe that, by H\"{o}lder's inequality,
\begin{eqnarray*}
|g_{f}(t)|&=&\rho(\nu(t))\,\nu'(t)\,\left|\int_0^{\nu(t)}
f'(z)\rd z\right|\,\le\,\rho(\nu(t))\,\nu'(t)\,\left(\int_0^{\nu(t)}\rd z
\right)^{1/p^*}\,\|f'\|_{L_p(\bbR_+)}\\
&= & h_{0,p^{\ast}}(t)\,\|f'\|_{L_p(\bbR_+)}. 
\end{eqnarray*}

We next prove {\it (ii)}. 
Obviously $g_f(0)=0$ since $\nu(0)=0$. We will prove the upper bound
on $C_{1,p}(\nu)$ only for $1<p<+\infty$ since the case of
$p=+\infty$ is even easier. We have 
\begin{eqnarray*}
g'_f(t)&=&f'(\nu(t))\,\rho(\nu(t))\,\left(\nu'(t)\right)^2
  +f(\nu(t))\,\left(\rho(\nu(t))\,\nu'(t)\right)'\\
&=&f'(\nu(t))\,\left(\nu'(t)\right)^{1/p}\,
  \rho(\nu(t))\,\left(\nu'(t)\right)^{1+1/p^*}
 +\left(\rho(\nu(t))\,\nu'(t)\right)'\,
  \int_0^{\nu(t)} f'(z)\rd z.
\end{eqnarray*}
Therefore
\begin{eqnarray*}
  \|g'_f\|_{L_p({[0,1)})} &\le& \left(\int_0^1|f'(\nu(t))|^p\,
    \nu'(t)\,\left(\rho(\nu(t))\,(\nu'(t))^{1+1/p^*}
    \right)^p\rd t\right)^{1/p}\\
&&+\left(\int_0^1|(\rho(\nu(t))\,\nu'(t))'|^p\,
  \left|\int_0^{\nu(t)} f'(z)\rd z\right|^p\rd t\right)^{1/p}.
\end{eqnarray*}
Clearly
\[
\int_0^1\left|f'(\nu(t))\right|^p\,\nu'(t)\rd t
=\int_0^\infty\left|f'(x)\right|^p\rd x=\|f'\|_{L_p(\bbR_+)}^p
\]
and, again by H\"{o}lder's inequality,
\[\left|\int_0^{\nu(t)} f'(z)\rd z\right|=\left|\int_0^\infty f'(z)\,1_{\{\nu(t)\ge z\}}\rd z\right|\le
\|f'\|_{L_p(\bbR_+)}\,(\nu(t))^{1/p^*}.
\]
Hence indeed,
\begin{eqnarray*}
\|g'\|_{L_p(0,1)}  
  &\le&\left(\int_0^1|f'(\nu(t))|^p\,\nu'(t)\rd t\right)^{1/p}\,
  \esup_{t\in[0,1)}\rho(\nu(t))\,(\nu'(t))^{1+1/p^*}\\
  &&+\left(\int_0^1|\left(\rho(\nu(t))\,\nu'(t)\right)'|^p\,
\|f'\|_{L_p(\bbR_+)}^{{p}}\,(\nu(t))^{p/p^*}\rd t \right)^{1/p}\\
&\le&\|f'\|_{L_p(\bbR_+)}\,\|\rho(\nu(\cdot))\,(\nu'(\cdot))^{1+1/p^*}
  \|_{L_\infty (0,1)}\\
&&  +\|f'\|_{L_p(\bbR_+)}\,
    \left\|\left(\rho(\nu(\cdot))\,\nu'(\cdot)\right)'\,
    (\nu(\cdot))^{1/p^*} \right\|_{L_p(0,1)}\\
 & = & \|f'\|_{L_p(\bbR_+)}\,\left(\|h_{1,p^*}\|_{L_{\infty}(0,1)}
+\|h_{2,p^*}\|_{L_p(0,1)}\right). 
\end{eqnarray*}
This completes the proof. 
\end{proof}

We give an example for an application of Theorem~\ref{thm:main}.

\begin{example}\label{ex:6} \rm
Let $D=\bbR_+$ and $\rho(x)=(c-1) (1+x)^{-c}$ for sufficiently large
$c$. We use Theorem~\ref{thm:main} for 
\[\nu_b(t)\,=\,\frac{1}{(1-t)^b}-1\quad\mbox{for\ }b\,>\,0.
\]
Then
\[
 \nu_b'(t)\,=\,\frac{b}{(1-t)^{b+1}}\quad\mbox{and}\quad
 \rho(\nu_b(t))\,=\,(c-1)\,(1-t)^{b\,c}.
\]
Therefore, for $p=+\infty$ we have 
\[
h_{1,1}(t)\,=\,b^2\,(c-1)\,(1-t)^{b\,(c-2)-2 }\quad\mbox{and}\quad
\|h_{1,1}\|_{L_\infty(0,1)}\,=\,b^2\,(c-1)
\]
for 
\[
b\,\ge\,\frac2{c-2}.
\]
Otherwise, $h_{1,1}$ is unbounded. As for $h_{2,1}$, we have 
\[\left(\rho(\nu_b(t))\,\nu_b'(t)\right)'\,=\,b\,(c-1)\,(b\,(c-1)-1)\,
(1-t)^{b(c-1)-2}
\]
and
\begin{eqnarray*}
h_{2,1}(t) &=&b\,(c-1)\,(b\,(c-1)-1)(1-t)^{b\,(c-1)-2}\,\left(\frac1{(1-t)^b}
-1\right)\\
&=&b\,(c-1)\,(b\,(c-1)-1)\left((1-t)^{b\,(c-2)-2}-
(1-t)^{b\,(c-1)-2}\right).
\end{eqnarray*}
Hence,
\begin{eqnarray*}
\|h_{2,1}\|_{L_\infty(0,1)} & = &  \esup_{x \in (0,1)} b\,(c-1)\,(b\,(c-1)-1)\left(x^{b\,(c-2)-2}-
x^{b\,(c-1)-2}\right) \\
& = & h_{2,1}\left(1-\left(\frac{b(c-2)-2}{b(c-1)-2}\right)^{1/b}\right)\\
& = & \frac{b^2 (c-1) (b(c-1)-1)}{b(c-1)-2} \left(\frac{b(c-2)-2}{b(c-1)-2}\right)^{c-2-\frac{2}{b}},
\end{eqnarray*}
and for $b >2/(c-2)$ we have $$C_{1,\infty}(\nu_b) \le b^2 (c-1)
\left(1+\frac{b(c-1)-1}{b(c-1)-2} \left(\frac{b(c-2)-2}{b(c-1)-2}\right)^{c-2-\frac{2}{b}}\right).$$
In view of what follows we mention that here for the
change of variables $\nu(t)=b\,\Phi_\rho^{-1}(t)$, we have
$\|h_{0,p^*}\|_{L_\infty(B)}=\|h_{1,p^*}\|_{L_\infty(B)}=
\|h_{2,p^*}\|_{L_p(B)}=+\infty$ for any $p>1$.   
\end{example}

\subsection{Case of $D=\bbR$}
In this case, we consider $\nu:(-1/2,1/2)\to\bbR$ that is increasing,
twice differentiable, and for which
\[\nu(0)\,=\,0\quad\mbox{and}\quad
\lim_{t\to\pm1/2}\nu(t)\,=\,\pm\infty. 
\]
To simplify the presentation, we also assume that 
the function $\rho$ is even and $\nu$ is odd, i.e.,
\begin{equation}\label{ass:1}
\rho(-x)\,=\,\rho(x)\quad\mbox{and}\quad \nu(-t)\,=\,-\nu(t),
\end{equation}
since then it is sufficient to consider only positive values of $t$.
However, if this assumption is not satisfied then one can treat the problem
as two subproblems: one with the domain $D=[0,+\infty)$ and the other with
$D=(-\infty,0]$. 

The following result is a version of Theorem~\ref{thm:main} adapted to the case $D=\bbR$. Due to \eqref{ass:1}, the proposition follows easily from the proof of Theorem \ref{thm:main}, and hence its proof is omitted.

\begin{theorem}\label{prop:main} 
If $\|h_{1,p^*}\|_{L_\infty(-1/2,1/2)}<+\infty$ and 
$\|h_{2,p^*}\|_{L_p(-1/2,1/2)}<+\infty$, then, 
for every $f\in F_{1,p}$, the corresponding
function $g_f$ belongs to $W_{1,p}$ and 
\[
C_{1,p}(\nu)\,\le\,\|h_{1,p^*}\|_{L_\infty(-1/2,1/2)}+
\|h_{2,p^*}\|_{L_{p}(-1/2,1/2)}
\]
with $h_{1,p^*}$ and $h_{2,p^*}$ defined by \eqref{H1} and \eqref{H2}, respectively. 
\end{theorem}

\subsection{Special Change of Variables}
We propose to use the following change of variables. 
If $D=\bbR_+,$ define $\nu=\nu_a:[0,1) \rightarrow \bbR$ by 
\begin{equation}\label{special}
\nu(t)\,=\,\nu_a(t)\,=\, a\,\Phi_\rho^{-1}(t). 
\end{equation}
If $D=\bbR,$ define  $\nu=\nu_a: (-1/2,1/2) \rightarrow \bbR$ by
\begin{equation}\label{special1}
\nu(t)\,=\,\nu_a(t)\,=\, a\,\Phi_\rho^{-1}\left(t+\frac{1}{2}\right), 
\end{equation}
where in both cases $a\ge 1$ and, as before, $\Phi_\rho^{-1}$ is the inverse
of the CDF for $\rho$.

\begin{proposition}\label{prop:aPhi}
Let $\nu_a$ be given by \eqref{special} or
\eqref{special1}, respectively. Then
\begin{eqnarray*}
\|h_{0,p^*}\|_{L_\infty({B})} & = & a^{1+1/p^*}\,\esup_{y\in D}
\frac{\rho(a\,y)}{\rho(y)}\,|y|^{1/p^*}, \\
\|h_{1,p^*}\|_{L_\infty({B})} & = & a^{1+1/p^*}\,\esup_{y\in D}
  \frac{\rho(a\,y)}{(\rho(y))^{1+1/p^*}}\, , \ \mbox{ and }\\
\|h_{2,p^*}\|_{L_p({B})} & = &
a^{1+1/p^*}\,\left(\int_D\left|\left(\frac{\rho(a\,y)}{\rho(y)}\right)'
\,\frac{y^{1/p^*}}{\rho(y)}\,\right|^p\,\rho(y)\rd y\right)^{1/p},
\end{eqnarray*}
where, as before, $B=[0,1)$ if $D=\bbR_+$ and $B=(-1/2,1/2)$ if $D=\bbR$. 
\end{proposition}
\begin{proof}
The result follows from a change of variables. If $D=\bbR_+$, then 
replace $t$ by $\Phi_\rho(y)$ and use the fact that 
\[
\nu'_a(t)\,=\,\frac{a}{\rho(\Phi_\rho^{-1}(t))}\,=\,
\frac{a}{\rho(y)}\quad\mbox{and}\quad
\rd t\,=\,\rho(y)\rd y.
\]
If $D=\bbR$, then replace $t$ by $\Phi_\rho(y)-\frac{1}{2}$ and
use the fact that
\[\nu'_a(t)\, = \, \frac{a}{\rho(\Phi_\rho^{-1}(t+1/2))}=\frac{a}{\rho(y)}
\quad\mbox{and}\quad \rd t\,=\,\rho(y)\rd y.
\]
This completes the proof.
\end{proof}

\begin{remark}\rm Note that for $p=1$, the exponent $1/p^*=\infty$.
Hence, for $a=1$, we have
\[\|h_{0,\infty}\|_{L_\infty(B)}\,=\,\|h_{1,\infty}\|_{L_\infty(B)}
\,=\,1\quad\mbox{and}\quad \|h_{2,\infty}\|_{L_{1}(B)}\,=\,0.
\]
This coincides with Part {\it (i)} of Proposition~\ref{prop:std}.
\end{remark}

\section{Particular Weights}
We illustrate Proposition \ref{prop:aPhi} for exponential and Gaussian
densities $\rho.$

\subsection{Exponential $\rho$}
We 
consider $D=\bbR_+$ and 
\[
\rho(x)\,=\,\frac1\lambda\,\exp\left(-\frac{x}{\lambda}\right)\quad
\mbox{for}\quad \lambda\,>\, 0.
\]
Clearly 
\[
\|I_{1,\rho}\|\,=\,
\left(\frac\lambda{p^*}\right)^{1/p^*}.   
\]
We use the change of variables
\begin{equation}\label{eqchangerho}
\nu_a(t)\,=\,a\, \Phi_{\rho}^{-1}(t)\,=\,-\lambda\, a\, \ln(1-t).
\end{equation}
It is easy to see that the norm of $h_{0,p^*}$ is infinite for
$a\le1$. 
Assume that $a>1$. Then 
\[\|h_{0,p^*}\|_{L_\infty (0,1)}\,=\,a^{1+1/p^*}\,
\left(\frac\lambda{{\mathrm{e}}\,p^*\,(a-1)}\right)^{1/p^*}.
\]
We also have 
\[\|h_{1,p^*}\|_{L_\infty(0,1)}\,=\,
  a^{1+1/p^*} \lambda^{1/p^*}\,
\sup_{y\in\bbR_+}\exp\left(-y\,\frac{a-(1+1/p^*)}\lambda\right),
\]
and hence 
\[\|h_{1,p^*}\|_{L_\infty{(0,1)}}\,=\,
\left\{\begin{array}{ll} a^{1+1/p^*}\lambda^{1/p^*} & \mbox{if\ }
a\,\ge\,1+\tfrac1{p^*},\\
+\infty &\mbox{otherwise,}
\end{array}\right.
\]
for any $p\in(1,+\infty]$.

In order to analyze $h_{2,p^*},$ we consider first $p=+\infty$. It is easy to see that then 
\[\|h_{2,1}\|_{L_\infty(0,1)}\,=\,
  \frac{a^2(a-1) \lambda}{(a-2) \mathrm{e}}\ \ \ \mbox{for $a>2$}
\]
and
\[C_{1,\infty}(\nu_a)\le a^2 \lambda \left(1+\frac{a-1}{(a-2) {\rm e}}\right)
\ \ \ \mbox{for $a>2$.}
\]
The upper bound on $C_{1,\infty}(\nu_a)$ is minimal for
\begin{equation}\label{eq:a*_p_infty}
a^*\,=\,2+\frac{4}{\sqrt{17+16\,{\rm e}}+1}\,=\,2.4557\ldots
\quad\mbox{and then}\quad C_{1,\infty}(\nu_{a^*})\,
\le\, \lambda \times 13.1172\ldots.
\end{equation}

Consider now finite $p>1$. Observe that 
\[\left(\frac{\rho(a\,y)}{\rho(y)}\right)'\,=\,
-\frac{a-1}\lambda\,\exp\left(-y\,\frac{a-1}\lambda\right)
\]
which results in
\[\|h_{2,p^*}\|_{L_p(0,1)}\,=\,
\frac{a^{1+1/p^*}\,(a-1)}{\lambda^{1/p}}\,
\left(\int_0^\infty y^{p-1}\,\exp\left(-y\,\frac{p\,(a-2)+1}\lambda
\right)\rd y\right)^{1/p},
\]
where we used the fact that $p/p^*=p-1$.

Suppose for the rest of this section  that $p$ is an integer.
Using integration by parts $p-1$ times, one can show that for any $c>0$, 
\[\int_0^\infty y^{p-1}\,\exp(-y\,c)\rd y\,=\,\frac{(p-1)!}{c^p},
\]
and conclude that 
\[\|h_{2,p^*}\|_{L_p(0,1)}\,=\,\frac{(a-1)\,a^{1+1/p^*}}
{p\,(a-2)+1}\, \lambda^{1/p^*}\,((p-1)!)^{1/p}, 
\]
whenever $a>1+1/p^*$. Finally,
\begin{equation}\label{help1}
C_{1,p}(\nu_a)\,\le\,a^{1+1/p^*}\,\lambda^{1/p^*}\,
\left(1 + \frac{(a-1)\,((p-1)!)^{1/p}}{p\,(a-2)+1}\right)\,<\,+\infty
\end{equation}
for $a>1+1/p^*$.

The upper bound on $C_{1,p}(\nu_a)$ is minimal for 
\begin{eqnarray*}
a^* & = &  2p (2p-1)^2 +(7 p^2-6p+1) ((p-1)!)^{1/p}\\
& &+\frac{\sqrt{p-1} ((p-1)!)^{1/(2p)} \sqrt{4 p^2 (2p-1)^2+(17 p^3-19 p^2+7 p -1) ((p-1)!)^{1/p}}}{2p (2p-1) (p+((p-1)!)^{1/p})}.
\end{eqnarray*}
For example, for $p=p^*=2$, 
\[a^*\,=\,\frac{53+\sqrt{217}}{36}\ =\ 1.8814\ldots
\]
which results in the upper bound
\[C_{1,2}(\nu_{a^*})\, \le\, \frac{(5+\sqrt{217})\,(53+\sqrt{217})^{3/2}}
{144\, (\sqrt{217}-1)}\,
\sqrt{\lambda}\, =\, \sqrt{\lambda}\,\times\, 5.5624\ldots .
\]

\begin{test}\label{Test1} \rm
Consider $f(x)=x$ and 
$\rho(x)=\exp(-x)$. Clearly, $f$ is in our space for $p=+\infty$ 
and $\|f\|_{F_{1,\infty}}=1$.
The value of the integral is $1$, and we use the change of variables 
$\nu_a$ as outlined above. Then the corresponding integrand is equal to 
\[g_{f,\nu_a}(t)\,=\,-a^2\,(\ln(1-t))\,(1-t)^{a-1}.
\]
Note that for $a=1$ the function $g_{f,\nu_a}(t)=-\ln(1-t)$ has a singularity at $t=1,$
for $1<a\le 2$ its derivative has a singularity at $t=1,$ and for $a>2$ the singularity
of $g_{f,\nu_a}'$ is removed and $g_{f,\nu_a}\in W_{1,\infty}$.

In the following table, we compare the integration errors of the midpoint rule
with $n$ samples applied to $g_{f,\nu_a}$ with $a=a^*$  as in \eqref{eq:a*_p_infty}, $a=1.5,$ and $a=1$. 
\[\begin{array}{lccc}
  n  &  a=a^*=2.4557\ldots & a=1.5 &  a=1 \\
\hline\\
10   & 4.353949E-03 & 1.118346E-02 & 3.424093E-02 \\
10^2 & 3.471053E-05 & 6.488305E-04 & 3.461569E-03 \\
10^3 & 2.958141E-07 & 3.029058E-05 & 3.465319E-04 \\
10^4 & 2.707151E-09 & 1.271297E-06 & 3.465694E-05 \\
10^5 & 2.596101E-11 & 5.015712E-08 & 3.465732E-06 \\ 
\end{array}
\]
\end{test}

\subsection{Gaussian $\rho$}
Consider $D=\bbR$ and 
\[
\rho(x)\,=\,\frac1{\sigma\,\sqrt{2\,\pi}}\,
\exp\left(-\frac{x^2}{2\,\sigma^2}\right)\quad\mbox{for\ }\sigma\,>\,0.
\]
Then
\[
\|I_{1,\rho}\|\,=\,\left(2^{1-p^*}\,\int_0^\infty\left(1-
   {\rm erf}\left(\frac{z}{\sqrt{2}\,\sigma}\right)\right)^{p^*}
     \rd z\right)^{1/p^*}\quad\mbox{for\ }p>1,
\]
where erf is the {\it Gauss error function} defined as ${\rm erf}(x) =
\tfrac{2}{\sqrt{\pi}} \int_0^x {\rm e}^{-t^2} \rd t$, and
\[\|I_{1,\rho}\|\,=\,1\quad\mbox{for\ }p=1.
\]
Note that $1-{\rm erf}(x)= \tfrac{2}{\sqrt{\pi}} \int_x^{\infty} {\rm e}^{-t^2} \rd t \le {\rm e}^{-x^2}$ for $x >0$ and 
hence $\|I_{1,\rho}\|< +\infty$ for all $p\ge 1$.

By Proposition~\ref{prop:aPhi}, for
\[
  \nu(t)\,=\,\nu_a(t)\,=\,a\,\Phi_\rho^{-1}(t+1/2)\quad\mbox{for\ }
t\,\in\,(-1/2,1/2),
\]
we get
\[ \|h_{0,p^*}\|_{L_\infty(-1/2,1/2)}\,=\,a\,
\left(\frac{a\,\sigma}{\sqrt{{\rm e}\,p^*\,(a^2-1)}}\right)^{1/p^*}.
\]
Moreover 
\[\| h_{1,p^*}\|_{L_\infty (-1/2,1/2)}
\,=\,
 a^{1+1/p^*} (\sigma \sqrt{2 \pi})^{1/p^*} \esup_{y\in\bbR} \exp\left(-\frac{y^2}{2\sigma^2}\left(a^2-\left(1+\frac{1}{p^*}\right)\right)\right).
\]
 Hence 
\[
\|h_{1,p^*}\|_{L_\infty (-1/2,1/2)}\,=\,\left\{\begin{array}{ll} \infty
&\mbox{if\ }a^2\,<\,1+\tfrac1{p^*},\\[0.5em]
c_1(a) &\mbox{if\ }a^2\,\ge\,1+\tfrac1{p^*},
\end{array}\right.
\]
where
\[c_1(a)\,=\,a^{1+1/p^*}\,(\sigma\,\sqrt{2\,\pi})^{1/p^*}.\]
Again, by Proposition \ref{prop:aPhi} we get 
\begin{eqnarray}\label{eq:normh2} 
 \lefteqn{\norm{h_{2,p^*}}_{L_p (-1/2,1/2)}=a^{1+1/p^*} \left(\int_{-\infty}^\infty \abs{\left(\frac{\rho (ay)}{\rho (y)}\right)'}^p (\rho(y))^{1-p} \abs{y}^{p/p^*}\rd y\right)^{1/p} 
 }\nonumber\\
 &=& \frac{a^{1+1/p^*}}{(\sigma \sqrt{2\pi})^{1/p-1}}\nonumber\\
 &&\times\left(\int_{-\infty}^\infty\left[\exp\left(-\frac{y^2}{2\sigma^2}(a^2-1)\right)\abs{\frac{y}{\sigma^2}(a^2-1)}
 \right]^p \exp\left(-\frac{y^2(1-p)}{2\sigma^2}\right) \abs{y}^{p/p^*}\rd y\right)^{1/p}\nonumber\\
 &=& \frac{a^{1+1/p^*}}{(\sigma \sqrt{2\pi})^{1/p-1}} \frac{a^2-1}{\sigma^2} 
 \left(\int_{-\infty}^\infty \exp\left(-\frac{y^2}{2\sigma^2}(a^2-1)p\right)
 \exp\left(-\frac{y^2 (1-p)}{2\sigma^2}\right)\abs{y}^{p+p/p^*}\rd y\right)^{1/p}\nonumber\\
  &=& \frac{a^{1+1/p^*}(a^2-1)}{(\sqrt{2\pi})^{1/p-1}\sigma^{1/p+1}} 
 \left(\int_{-\infty}^\infty \exp\left(-\frac{y^2}{2\sigma^2}((a^2-2)p+1)\right)
 \abs{y}^{p+p/p^*}\rd y\right)^{1/p}.
\end{eqnarray}
For $p$ tending to $+\infty$ (and $p^*$ tending to 1)
we see that 
\[\frac{a^{1+1/p^*}(a^2-1)}{(\sqrt{2\pi})^{1/p-1}\sigma^{1/p+1}}\rightarrow c_2(a)\,:=\,\frac{\sqrt{2 \pi}\, a^2\,(a^2-1)}{\sigma}.\]
Furthermore, we write 
\begin{eqnarray*}
&&\int_{-\infty}^\infty \exp\left(-\frac{y^2}{2\sigma^2}((a^2-2)p+1)\right)
 \abs{y}^{p+p/p^*}\rd y\\
&& =\,\int_{-\infty}^\infty \left[
   \exp\left(-\frac{y^2}{2\sigma^2}\left((a^2-2)+1-\frac{1}{p^*}\right)\right)
 \abs{y}^{1+1/p^*}\right]^p\rd y.
\end{eqnarray*}
From this we conclude that
\[\|h_{2,1}\|_{L_\infty(-1/2,1/2)}\,=\,c_2(a)\,
\max_{z\in\bbR}z^2\,\exp\left(-\frac{z^2}{2\,\sigma^2}\,(a^2-2)\right)
\,=\,c_2(a)\,{\rm e}^{-1}\,\frac{2\,\sigma^2\,}{a^2-2}
\]
for $a^2>2$. Hence 
\[C_{1,\infty}(\nu_a)\,\le\,\sqrt{2\,\pi}\sigma\,a^2\,\left(1+\frac2{{\rm e}}\,
\frac{a^2-1}{a^2-2}\right). 
\]
It is easy to check that the upper bound on $C_{1,\infty}(\nu_a)$, as a function of $a$, 
attains its minimum at 
\begin{equation}\label{star2}
  a^*\,=\sqrt{ 2 + \frac{2}{\sqrt{2+\rm e}} }\,=\,1.70902\ldots
  \quad\mbox{and then}\quad C_{1,\infty}(a^*)\,=\,\sigma\, \times\, 18.5582\ldots
  \quad\mbox{for\ }p=+\infty.
\end{equation}

Returning to \eqref{eq:normh2}, and using $p(1+1/p^*)=2p-1$,
we have for $p<+\infty$,
\begin{eqnarray*}
 \norm{h_{2,p^*}}_{L_p (-1/2,1/2)}
 &=&\frac{a^{1+1/p^*}(a^2-1)}{(\sqrt{2\pi})^{1/p-1}\sigma^{1/p+1}} 
 \left(\int_{-\infty}^\infty \exp\left(-\frac{x^2}{2\sigma^2}((a^2-2)p+1)
 \right)
 \abs{x}^{p+p/p^*}\rd x\right)^{1/p}\\
 &=& \,c_2(a)\,\left(\frac1{\sqrt{2\,\pi}\,\sigma\, a}\,
\int_{-\infty}^{\infty}|x|^{2\,p-1}\,\exp\left(-\frac{x^2}{2\,\sigma^2}\,
(p\,(a^2-2)+1)\right)\rd x\right)^{1/p}.
\end{eqnarray*}

This means that $\|h_{2,p^*}\|_{L_p(-1/2,1/2)}<+\infty$ if and only if
$p(a^2-2)+1>0$, i.e., 
\[a^2\,>\,1+\frac1{p^*}.  
\]
By the change $y=\frac{x}{\sigma}\,\sqrt{p(a^2-2)+1}$ we get that 
\[\|h_{2,p^*}\|_{L_p(-1/2,1/2)}\,=\,c_3(a)\,\left(\frac1{\sqrt{2\,\pi}}
\,\int_{-\infty}^{\infty}|y|^{2\,p-1}\,\exp\left(-\frac{y^2}{2}\right)
\rd y\right)^{1/p},
\]
where
\[c_3(a)\,=\,c_2(a)\,\frac{\sigma^{2-1/p}}{(p\,(a^2-2)+1)\,a^{1/p}}\,=\,
\frac{\sqrt{2 \pi} \, a^{1+1/p^*}\,(a^2-1) \, \sigma^{1/p^*}}{(p\,(a^2-2)+1)}.
\]
It is well known that for natural numbers $p$,
\[\frac1{\sqrt{2\,\pi}}\,\int_{-\infty}^{\infty}|y|^{2\,p-1}\,
\exp\left(-\frac{y^2}{2}\right)\rd y\,=\,\frac1{\sqrt{2\,\pi}}\,2^{p}\,(p-1)!\, .
\]
Hence, for natural numbers $p$,
\[\|h_{2,p^*}\|_{L_p(-1/2,1/2)}\,=\,c_3(a)\,\frac{2\,((p-1)!)^{1/p}}{(2\,\pi)^{1/(2p)}},
\]
and, for $a^2>1+\frac{1}{p^*}$,
\begin{eqnarray*}
C_{1,p}(\nu_a)&\le&c_1(a)+c_3(a)\,\frac{2\,((p-1)!)^{1/p}}{(2\,\pi)^{1/(2p)}}\\
&=&a^{1+1/p^*}\,(\sigma \, \sqrt{2\,\pi})^{1/p^*}+ 
\frac{\sqrt{2 \pi} \, a^{1+1/p^*}\,(a^2-1) \, \sigma^{1/p^*}}{p\,(a^2-2)+1}\,
\frac{2\,((p-1)!)^{1/p}}{(2\,\pi)^{1/(2p)}}\\
&=& a^{1+1/p^*}\,(\sigma \, \sqrt{2\,\pi})^{1/p^*}\,\left(1
+\frac{ 2\,(a^2-1)\,((p-1)!)^{1/p}}{p\,(a^2-2)+1}\right) < +\infty. 
\end{eqnarray*}

For example, if $p=p^*=2$, then
\[C_{1,2}(\nu_a)\,\le\,a^{3/2} \sqrt{\sigma \sqrt{2 \pi}}
\left(2+\frac{1 }{2 a^2 -3}\right).
\]
As a function of $a$, the upper bound on 
$C_{1,2}(\nu_a)$ attains its minimum at $a^*=3/2$ and then
\[C_{1,2}(\nu_{a^*})\,\le\,2 \sqrt{6 \sqrt{2 \pi}}\,
\sqrt{\sigma} =\sqrt{\sigma} \times 7.7562\ldots.
\]

If $p$ is not a natural number, we can estimate 
\begin{eqnarray*}
\frac1{\sqrt{2\,\pi}}\,\int_{-\infty}^{\infty}|y|^{2\,p-1}\,
\exp\left(-\frac{y^2}{2}\right)\rd y&\,\le\,&C\, +\,
\frac1{\sqrt{2\,\pi}}\,\int_{-\infty}^{\infty}|y|^{2\,\lceil p \rceil-1}\,
\exp\left(-\frac{y^2}{2}\right)\rd y\\
&\,=\,&C\, +\,
\frac1{\sqrt{2\,\pi}}\,2^{\lceil p \rceil}\,(\lceil p \rceil-1)!\,
\end{eqnarray*}
for some absolute constant $C>0$. Then, $\|h_{2,p^*}\|_{L_p (-1/2,1/2)}$ 
and $C_{1,p}(\nu_a)$ can be estimated accordingly.

\medskip
\begin{test}\label{Test2}
\rm 
Consider integration of the function $f(x)=|x|,$ similar to
Numerical Test \ref{Test1}.
Then the integral to be computed is
\[\int_{-\infty}^\infty\frac{|x|}{\sigma\sqrt{2\pi}}\,
\exp\left(-\frac{x^2}{2\sigma^2}\right)\rd x 
=\sqrt\frac 2\pi.
\]
In this case, our change of variables is defined as
$x=\nu_a(t):=a\sqrt{2}\,\mathrm{erfinv}(2t),$ where 
$a\ge 1,$ and $\mathrm{erfinv}$ is the inverse of the erf function.
This results in the integral
$$ \int_{-1/2}^{1/2}\,g_{f,\nu_a}(t)\,\mathrm d t\qquad\mbox{with}\qquad
     g_{f,\nu_a}(t)=a\;|\nu(t)|\exp\left(-\frac{\nu^2(t)}2\left(1-\frac 1{a^2}\right)\right). $$
Note that $g_{f,\nu_a}\in W_{1,\infty}$ for $a>\sqrt 2.$

In the following table, we compare the absolute integration  errors of the midpoint rule
with $n$ samples applied to $g_{f,\nu_a}$ with $a=a^*$  as in \eqref{star2}, $a=\sqrt 2,$ and $a=1$. 

\[\begin{array}{lccc}
  n  &  a=a^*=1.70902\ldots  &  a=\sqrt{2} & a=1 \\
\hline\\
10   & 6.044012E-03 & 8.241433E-03 & 2.734692E-02 \\
10^2 & 1.096395E-04 & 1.157302E-03 & 2.106636E-03 \\
10^3 & 1.200953E-06 & 6.968052E-05 & 1.770735E-04 \\
10^4 & 1.217389E-08 & 3.509887E-06 & 1.559486E-05 \\ 
10^5 & 1.219963E-10 & 1.642630E-07 & 1.409149E-06 \\
\end{array}
\]
\end{test}

\section{$\gamma$-Weighted Spaces and MDM}\label{Sect:MDM}

We show in this section how the algorithms derived via the change of
variables can be used in efficient implementation of the
{\em Multivariate Decomposition Method}  ({\em MDM} for short)
that was introduced in \cite{KuSlWaWo2011} (see also \cite{KuNuPlSlWa2017}). 
Our presentation follows \cite{Wa2013,Wa2014}, where more details can be found.

Consider the following $\gamma$-weighted space 
$\calF_\gamma=\calF_{\gamma,d,p,q}$ of functions $f:D\to\bbR$ with
mixed first order partial derivatives bounded in $L_p(D^d)$ 
(as before) and for which the following norm of $\calF_\gamma$ is finite,
\[\|f\|_{\calF_\gamma}
\,:=\,\left[\sum_\setu \gamma_\setu^{-q}\,
  \|f^{(\setu)}(\cdot_\setu;\bszero)\|_{L_p(D^\setu)}^q\right]^{1/q}
  \,<\,\infty.
\]
Here, $q\in[1,\infty]$, the summation above is with respect to all subsets
\[
\setu\,\subseteq\,\{1,2,\dots,d\},
\]
$f^{(\setu)}$ denotes
$\prod_{j\in\setu}\frac{\partial}{\partial x_j}\,f$, the vector 
$(\bsx_\setu;\bszero)$ is the $d$-dimensional vector $\bsx=(x_1,x_2,\ldots,x_d)$ with the $j$-th component $x_j$ set to zero whenever $j\notin\setu$, 
and $\gamma_\setu$ are given positive numbers quantifying
importance of the subsets $\setu$. For $\setu=\emptyset$, we set
$f^{(\emptyset)}\equiv f(0)$ and $\gamma_\emptyset=1$. 
For simplicity, we consider only so-called {\em product weights},
introduced in \cite{SlWo98}, of the form
\[
\gamma_\setu\,=\,\prod_{j\in\setu}\gamma_j, 
\]
where $\{\gamma_j\}_{j\ge 1}$ is a given non-increasing sequence of positive numbers.

\begin{remark} \rm
The results can easily be extended to functions with infinitely many
($d=+\infty$) variables. Then the summation is with respect to
all finite subsets $\setu$ of $\bbN$.
\end{remark}

It is well known that any $f\in\calF_\gamma$ has a unique
{\em anchored decomposition} of the form
\[f(\bsx)\,=\,\sum_\setu f_\setu(\bsx),
\]
where, for $\setu\not=\emptyset$, 
$f_\setu$ depends only on $x_j$ with $j\in\setu$ and is anchored at zero, i.e., 
$$\mbox{$f_\setu(\bsx)=0$ if there is $j \in \setu$ such that $x_j=0$.}$$ 
Moreover,
\[
f_\setu^{(\setu)}\,=\,f^{(\setu)}(\cdot_\setu;\bszero),\quad
\mbox{which implies that}\quad
\|f\|_{\calF_\gamma}\,=\,\left[\sum_\setu \gamma_\setu^{-q}\,
  \|f_\setu\|_{F_{\setu,p}}^q\right]^{1/q}\quad\mbox{and}\quad
  f_\setu\,\in\,F_{\setu,p},
\]
where the spaces $F_{\setu,p}$ are equivalent to $F_{d,p}$
for $d=|\setu|$ with the only difference that the functions in
$F_{\setu,p}$ depend on the variables $x_j$ with $j\in\setu$.
In particular, $F_{d,p}=F_{\{1,2,\ldots,d\},p}$. Therefore
\[
 I_{d,\rho}(f)\,=\,\sum_\setu I_{|\setu|,\rho}(f_\setu).
\]

Suppose now that
\begin{equation}\label{eq:qgamma}
 \left[ \sum_{j=1}^\infty \gamma_j^{q^*}\right]^{1/q^*}\,<\,\infty.
\end{equation}
Following \cite{PlWa2011}, one can show that for $\e>0$ there is a
set ${\rm Act}(\e)$ containing some $\setu$'s with cardinality
$|{\rm Act}(\e)|=O(\e^{-1})$ such that 
\[
\sum_{\setu\notin{\rm Act}(\e)}|I_{|\setu|,\rho}(f_\setu)|\,\le\,
\frac{\e}{2^{1/q^*}}\,\left\|\sum_{\setu\notin{\rm Act}(\e)} f_\setu
\right\|_{\calF_\gamma}\quad\mbox{and}\quad 
d(\e)\,:=\,\max_{\setu\in{\rm Act}(\e)} |\setu|\,=\,O\left(
\frac{\ln(1/\e)}{\ln(\ln(1/\e))}\right).
\]
As shown recently in \cite{GiWa2016}, the absolute constants
in the big-$O$ notations above are very small, see also Example \ref{ex:infty}. 
Hence it is enough to aproximate the integrands
$I_{|\setu|,\rho}(f_\setu)$ for $\setu\in{\rm Act}(\e)$ with the total
error bounded by
\[
\frac{\e}{2^{1/q^*}}\,\left\|\sum_{\setu\in{\rm Act(\e)}} f_\setu
\right\|_{\calF_\gamma}.
\]
This can be achieved by using cubatures $Q_{|\setu|,\rho,n_{\setu}}$
(see \eqref{dupa-2}) with appropriately chosen natural numbers
$n_\setu$ such that
\[
\left[\sum_{\setu\in{\rm Act}(\e)}\left(\gamma_\setu\,C_{|\setu|,p}(\nu)
  \,{\rm error}(Q_{|\setu|;n_\setu}\,;\,
  W_{|\setu|,p})\right)^{q^*}\right]^{1/q^*}
\le\,\frac{\e}{2^{1/q^*}},
\]
which follows from \eqref{dupa-3}. Note that 
\[
\gamma_\setu\,C_{|\setu|,p}(\nu)\,=\,\prod_{j\in\setu}
(\gamma_j\,C_{1,p}(\nu)),
\]
due to Proposition \ref{prop:univar}. Hence, as shown in, e.g.
\cite{Wa2013}, the sum of all $n_\setu$ is small,
\[
  \sum_{\setu\in{\rm Act}(\e)} n_\setu\,=\,O(\e^{-1}).
\]
This means that the corresponding MDM uses altogether $O(1/\e)$
function values to approximate $O(1/\e)$ integrals, each with
at most $d(\e)=O(\ln(1/\e)/\ln(\ln(1/\e)))$ variables.

We now illustrate this by the following special case.

\begin{example}\label{ex:infty} \rm
Let $D=\bbR_+$, $\rho(x)=\exp(-x/\lambda)/\lambda$, $\gamma_j=1/j^{\beta}$ with some positive $\beta$,
and $q=1$. Recall that then $\|I_{1,\rho}\|=(\lambda/p^*)^{1/p^*}$ and,
hence, we can take 
\[{\rm Act}(\e)\,=\,\left\{\setu\,:\, \|I_{1,\rho}\|^{|\setu|}\,\gamma_\setu
\,>\,\e\right\}\,=\,\left\{
\setu\,:\,\frac{(\lambda/p^*)^{|\setu|/p^*}}{\prod_{j\in\setu}j^\beta}\,>\,
\e\right\},
\]
for which
\[d(\e)\,=\,\max\left\{k\,:\,\frac{(\lambda/p^*)^{k/p^*}}
{(k!)^\beta}\,>\,\e\right\}.
\]
For instance for $p=p^*=2$, $\lambda=2$, and $\e=10^{-4}$ we have:

$$\begin{array}{l||l|l|l|l}
\beta & 2 & 3 & 4 & 5\\
\hline
d(\e) & 4 & 3 & 3 & 3
\end{array}$$
\end{example}

\begin{test}\label{Test3}{\rm
Let, similarly to Test \ref{Test1}, 
$D=\bbR_+$, $\rho(x)=\exp(-x)$, and define, for $d\ge 1$, 
$f_d (\bsx)=\prod_{j=1}^d x_j$ for $\bsx=(x_1,\ldots,x_d)\in\bbR_+^d$. 
The value of the integral is then 1 
and we have $\norm{f}_{\infty,d}=1$. Using the change of variables $\nu_a$ from \eqref{eqchangerho}, 
the corresponding integrand is
\[ g_{d,a}(\bst)\,=\,\prod_{j=1}^d \left(-a^2\,(\ln(1-t_{j}))\,(1-t_{j})^{a-1}\right).\]

As integration rules, we use lattice rules\footnote{We use lattice rules with generating vectors taken from the website of Frances Y.~Kuo. 
The generating vectors used in this example 
were generated for equal product weights $\gamma_j=1$, referred to as ``lattice-28001'' at
\texttt{http://web.maths.unsw.edu.au/\~{ }fkuo/lattice/index.html} .} 
(cf. \cite{SJ94}) with $n=2^k$, $k=10,11,12,\ldots,15$, points in $[0,1)^d$. 

We compare the absolute integration errors
for $a=a^*$ as in \eqref{eq:a*_p_infty}, $a=1.5$, and $a=1$, and $d=3$ and $d=4$. In view of Example \ref{ex:infty}, 
it is justified to concentrate on $d$ in this range. The results below indicate that the approach taken in this paper 
yields effective improvements over the standard change of variables.

\[
\begin{array}{lccc}
d=3 & & & \\
 \hline\
n  &    a=a^*=2.4557\ldots  &  a=1.5 & a=1  \\
\hline\\
2^{10} &  1.088922E-04 & 9.822891E-04 & 3.702216E-03\\
2^{11} &  3.972304E-05 & 3.220787E-04 & 1.725302E-02\\
2^{12} &  1.160874E-05 & 1.342448E-04 & 4.356446E-03\\
2^{13} &  3.077676E-06 & 1.812985E-04 & 7.678733E-03\\
2^{14} &  7.859542E-07 & 1.831772E-05 & 2.343056E-03\\ 
2^{15} &  2.052312E-07 & 2.369181E-06 & 3.260093E-03\\ 

\end{array}
\]

\[
\begin{array}{lccc}
d=4 & & & \\
 \hline\
n  &    a=a^*=2.4557\ldots &  a=1.5 & a=1  \\
\hline\\
2^{10} &  2.125335E-05 & 1.461441E-03 & 1.316315E-02\\
2^{11} &  1.092869E-05 & 7.043893E-04 & 3.040979E-02\\
2^{12} &  5.426712E-06 & 8.995698E-06 & 1.980920E-02\\
2^{13} &  6.603896E-06 & 2.797430E-04 & 1.999552E-02\\
2^{14} &  8.905416E-06 & 9.166504E-05 & 5.214687E-03\\ 
2^{15} &  2.602460E-06 & 6.881108E-05 & 3.043182E-04\\ 

\end{array}
\]

}
\end{test}

\section*{Acknowledgments}

P.~Kritzer, L.~Plaskota, and G.W.~Wasilkowski would like to thank the MATRIX institute in Creswick, VIC, Australia, and
its staff for supporting their stay during the program ``On the Frontiers of High-Dimensional Computation'' 
in June 2018. Furthermore, the authors thank the RICAM Special Semester Program
2018, during which parts of the paper were written.


\begin{small}
\noindent\textbf{Authors' addresses:}

\medskip
\noindent Peter Kritzer\\
Johann Radon Institute for Computational and Applied Mathematics (RICAM)\\
Austrian Academy of Sciences\\
Altenbergerstr.~69, 4040 Linz, Austria\\
E-mail: \texttt{peter.kritzer@oeaw.ac.at}

\medskip

\noindent Friedrich Pillichshammer\\
Institut f\"{u}r Finanzmathematik und Angewandte Zahlentheorie\\
Johannes Kepler Universit\"{a}t Linz\\
Altenbergerstr.~69, 4040 Linz, Austria\\
E-mail: \texttt{friedrich.pillichshammer@jku.at}

\medskip

\noindent Leszek Plaskota\\
Institute of Applied Mathematics and Mechanics\\
Faculty of Mathematics, Informatics, and Mechanics\\
University of Warsaw\\
Banacha 2, 02-097 Warsaw, Poland\\
E-mail: \texttt{leszekp@mimuw.edu.pl} 

\medskip

\noindent G. W. Wasilkowski\\
Computer Science Department, University of Kentucky\\
301 David Marksbury Building\\
329 Rose Street\\
Lexington, KY 40506, USA\\
E-mail: \texttt{greg@cs.uky.edu}
\end{small}

\end{document}